\documentclass[a4, 11pt] {amsart}
\usepackage{amssymb,times, amscd,amsmath,amsthm, xypic}
\usepackage{geometry}
\usepackage[all]{xy}
\usepackage[dvips]{graphicx, color}
\usepackage{hyperref}
\usepackage{setspace}
\numberwithin{equation}{section}
\newtheorem{thm}[equation]{Theorem}
\newtheorem{pro}[equation]{Proposition}

\newtheorem{cor}[equation]{Corollary}

\theoremstyle{definition}

\newtheorem{defn}[equation]{Definition}

\newtheorem{rem}[equation]{Remark}

\def\alp{\alpha}

\def \bC{\mathbb C}

\def\bR{\mathbb R}

\def\bZ{\mathbb Z}
\def\zz{\mathbb Z}
\def\bC{\mathbb  C}

\newcommand{\bb}[1]{\mbox{$\mathbb{#1}$}}
\def\op{\operatorname}

\author{Carmen Rovi}

\address{Department of Mathematics, Indiana University Bloomington, Rawles Hall, 831 E 3rd St, Bloomington, IN 47405}

\email{crovi@indiana.edu}

\author{Shoji Yokura$^*$}

\address{Department of Mathematics and Computer Science,
Graduate School of Science and Engineering,\\ Kagoshima University,
21-35 Korimoto 1-Chome,
Kagoshima 890-0065, Japan}
\email{yokura@sci.kagoshima-u.ac.jp}

\date{}

\thanks{(*) Partially supported by JSPS KAKENHI Grant Number 16H03936}

\title[Hirzebruch $\chi_y$-genera mod $8$ of fiber bundles]
{Hirzebruch $\chi_y$-genera modulo $8$ of fiber bundles\\
for odd integers $y$}
\usepackage{color}

\begin{document}

\maketitle

\begin{abstract} I. Hambleton, A. Korzeniewski and A. Ranicki have proved that the signature of a fibre bundle $F \hookrightarrow E \to B$ of closed, connected, compatibly oriented PL manifolds is always multiplicative mod 4, i.e. $\sigma(E)\equiv \sigma(F)\sigma(B) \, \op{mod} \, 4.$ In this paper we consider the Hirzebruch $\chi_y$-genera for odd integers $y$ for a smooth fiber bundle $F \hookrightarrow E \to B$ such that $E, F$ and $B$ are compact complex algebraic manifolds (in the complex analytic topology, not in the Zariski topology). In particular, if $y=1$, then $\chi_1$ is the signature $\sigma$. We show that the Hirzebruch $\chi_y$-genera of such a fiber bundle are always multiplicative mod $4$, i.e. $\chi_y(E) \equiv \chi_y(F) \chi_y(B) \, \op{mod} \, 4$. We also investigate multiplicativity mod $8$, and show that if $y \equiv 3 \,\, \op{mod}\,\, 4$, then $\chi_y(E) \equiv \chi_y(F) \chi_y(B) \, \op{mod} \, 8$ and that in the case when $y \equiv 1 \,\, \op{mod} 4$ the Hirzebruch $\chi_y$-genera of such a fiber bundle is multiplicative mod $8$ if and only if the signature is multiplicative mod $8$, and that the non-multiplicativity modulo $8$ is identified with an Arf-Kervaire invariant.
\end{abstract}


\section{Introduction}
The Hirzebruch $\chi_y$-genus $\chi_y(X)$ of a compact complex algebraic manifold $X$ was introduced by F. Hirzebruch \cite{Hir} (also see \cite{HBJ}) in order to extend his famous Hirzebruch-Riemann-Roch theorem to the parametrized setting. If $y=-1, 0, 1$, then these $\chi_y$-genera are respectively
\begin{itemize}
\item $\chi_{-1}(X) =\chi(X)$ the \emph{Euler-Poincar\'e characteristic},
\item $\chi_0(X) = \tau(X)$ the \emph{Todd genus},
\item $\chi_1(X)=\sigma(X)$ the \emph{signature},
\end{itemize}
 which are very important invariants in geometry and topology, and even in mathematical physics. 
 
The Euler-Poincar\'e characteristic is multiplicative for \emph{any} topological fiber bundle $F \hookrightarrow E \to B$, i.e. $\chi(E) = \chi(F)\chi(B)$ holds. The signature is in general not multiplicative for fibre bundles. S. S. Chern, F. Hirzebruch and J.-P. Serre \cite{CHS} proved that the signature is multiplicative for a fiber bundle under a certain monodromy condition, i.e. if the fundamental group $\pi_1(B)$ of the base space $B$ acts trivially on the cohomology group $H^*(F; \mathbb R)$ of the fiber space $F$.  Later  M. Atiyah \cite{At}, F. Hirzebruch \cite {Hir0} and K. Kodaira \cite{Ko} gave the first examples of fibre bundles with non-multiplicative signatures.

I. Hambleton, A. Korzeniewski and A. Ranicki  \cite{HKR}  showed that for a $PL$ fibre bundle $F \hookrightarrow E \to B$ of closed, connected, compatibly oriented $PL$ manifolds
$$\sigma(E) \equiv \sigma(F) \sigma(B) \op{mod} 4.$$
In \cite{Rov, Rov2} C. Rovi has shown that for a fiber bundle $F \hookrightarrow E \to B$ in the manifold context, if the action of $\pi_1(B)$ on $H^m(F, \zz)/ torsion \otimes \mathbb Z_4$ is trivial (where $\op{dim}_{\mathbb R} F = 2m$),
$$\sigma(E) \equiv \sigma(F) \sigma(B) \op{mod} 8.$$
In \cite{Rov, Rov2} C. Rovi has also shown that the non-multiplicativity of the signature modulo $8$ of a fibre bundle is detected by the 
$\zz_2$-valued Arf-Kervaire invariant of a certain quadratic form associated to the fiber bundle as above.

In \cite{Yo} S. Yokura has studied some explicit formulae of the Hirzebruch $\chi_y$-genera for complex fiber bundles $F \hookrightarrow E \to B$ where $F, E, B$ are compact complex algebraic manifolds, and 
he has observed that $\sigma(E) \equiv \sigma(F) \sigma(B) \op{mod} 4$ just like the above result of Hambleton-Korzeniewski-Ranicki.

In this paper we consider such mod $4$ and mod $8$ multiplicativity formulae of Hirzebruch $\chi_y$-genera for a fibre bundle $F \hookrightarrow E \to B$ of compact complex algebraic manifolds. A smooth proper map between compact complex algebraic manifolds is a locally trivial topological fibration (by Ehresmann's fibration theorem), thus becomes such a fiber bundle, and conceivably such fiber bundles arise in this way. 

The main result 
of this paper is stated in Theorem \ref{mod8}, which gives an overview of the non-multiplicative behaviour of $\chi_y$ genera modulo $8$ of a fiber bundle 
for odd values of $y$. It is interesting to note that when $y \equiv 1 \mod{4}$ 
the $\chi_y$-genera adopt the same non-multiplicativity behaviour as the signature, and therefore the obstruction for multiplicativity in this case (when $y \equiv 1 \mod{4}$) is detected by the Arf-Kervaire invariant of a certain quadratic form associated to the fiber bundle, which is described in Theorem \ref{mod8} below.

\begin{thm}\label{mod8} Let $F \hookrightarrow E \to B$ be a fiber bundle such that $F, E, B$ are compact complex algebraic manifolds. 

\begin{itemize}
\item[(a)] If $y \equiv 3 \, \op{mod} \, 4$, then $\chi_y(E) \equiv \chi_y(F)\chi_y(B) \, \op{mod} \, 8.$
\item[(b)] If $y \equiv 1 \, \op{mod} \, 4$, then $\chi_y(E) \equiv \chi_y(F)\chi_y(B) \, \op{mod} \, 8 \Longleftrightarrow \, \sigma(E) \equiv \sigma(F)\sigma(B) \, \op{mod} \, 8.$ Moreover
$$\chi_y(E) - \chi_y(F) \chi_y(B) \equiv 4 \textnormal{Arf}(W, \mu, h)  \pmod{8}.$$
where $(W, \mu, h)$ is a certain $\zz_2$-valued quadratic form associated to the fiber bundle. (For details see Theorem \ref{Arf-chiy} below.)
\end{itemize}
\end{thm}
The proof of Theorem \ref{mod8} will be subdivided into shorter statements. In particular,
(a) and the first part of (b)  
is proved in Theorem \ref{mod8-part1}. 
The second part of (b), i.e. the identification of the obstruction to multiplicativity with the Arf-Kervaire invariant 
is shown in Theorem \ref{Arf-chiy}.

There is another result which  follows from \cite[Theorem 3.5]{Rov} and Theorem 1.1 
:
\begin{thm}\label{cor1} Let $F \hookrightarrow E \to B$ be a fiber bundle such that $F, E, B$ are compact complex algebraic manifolds  with $\textnormal{dim}_{\bR} F =2m$. If the action of $\pi_1(B)$ on $H^m(F, \bZ)/torsion \otimes \bZ_4$ is trivial, then for \emph{any} odd integer $y$
$$\chi_y(E) \equiv \chi_y(F)\chi_y(B) \, \op{mod} \, 8.$$
\end{thm}

\begin{rem}
From Theorem \ref{cor1} we can immediately observe that \emph{if $B$ is simply connected}, such as smooth complex rational varieties and smooth Fano varieties, then for \emph{any} odd integer $y$
$$\chi_y(E) \equiv \chi_y(F)\chi_y(B) \, \op{mod} \, 8.$$
However, in this case the usual equality does hold:
$$\chi_y(E) = \chi_y(F)\chi_y(B) \,\in \zz.$$
This is because in this case the action of $\pi_1(B)$ is automatically trivial on $H^*(F,\mathbb Z)$, from which $\chi_y(E) = \chi_y(F)\chi_y(B)$ follows, as explained in \cite[Remark 4.6]{MS} (also see \cite{CMS1, CLMS1, CLMS2}).
\end{rem}

\section{Hirzebruch $\chi_y$-genera}
First we recall the  definition of the Hirzebruch $\chi_y$-genus \cite{Hir}. 

Let $X$ be a compact complex algebraic manifold. The
\emph{$\chi_{y}$-genus} of $X$ is defined by
$$\chi_{y}(X):=
\sum_{p\geq 0} \chi(X,\Lambda^{p}T^{*}X)y^p
= \sum_{p\geq 0} \left( \sum_{i\geq 0}
(-1)^{i}\op{dim}_{\mathbb C}H^{i}(X,\Lambda^{p}T^{*}X) \right)y^p\:.
$$

Thus the $\chi_y$-genus is the generating function of the Euler-Poincar\'e characteristic $\chi(X,\Lambda^{p}T^{*}X)$ of the sheaf $\Lambda^{p}T^{*}X$, which shall be simply denoted by $\chi^p(X)$:
$$\chi_{y}(X)= \sum_{p\geq 0} \chi^p(X)y^p.$$
Since $\Lambda^{p}T^{*}X =0$ for $p> \op{dim}_{\mathbb C}X$, $\chi_{y}(X)$ is a polynomial of degree at most $\op{dim}_{\mathbb C} X$. Note that for an ordinary product of spaces $\chi_y$ is multiplicative, i.e. $\chi_y(X \times Y) = \chi_y(X) \chi_y(Y)$.

Then we have the following ``generalized Hirzebruch-Riemann-Roch theorem" (abbr., gHRR):
$$
\chi_{y}(X)= \int_{X} T_{y}(TX)\cap [X]
\quad \in \bb{Q}[y],
$$
where $T_{y}(TX)$ is the generalized Todd class of the tangent bundle of $X$. $T_y(TX)$, which we will denote simply by $T_y(X)$, is defined as follows:
$$T_{y}(X):= \prod_{i=1}^{\op{dim} X} \left (\frac{\alpha_i(1+y)}{1-e^{-\alpha_i(1+y)}} -\alpha_i y \right ),
$$
where $\alpha_{i}$ are the Chern roots of the tangent bundle $TX$. That is, writing $c(X)$ as the total Chern class of $TX$,
$$c(X) = \prod_{i=1}^{\op{dim} X} (1 + \alp_i).$$
Note that the normalized
power series
\[Q_{y}(\alpha):= \frac{\alpha(1+y)}{1-e^{-\alpha(1+y)}} -\alpha y
\quad \in \bb{Q}[y][[\alpha]] \:. \]
specializes to
$$
Q_{-1}(\alp) = \:\displaystyle  1+\alpha,\quad
Q_0(\alp) = \displaystyle  \frac{\alpha}{1-e^{-\alpha}},\quad
Q_1(\alp) = \displaystyle   \frac{\alpha}{\tanh \alpha}
$$
Therefore $T_{y}(X)$ unifies the
following important characteristic cohomology classes of $X$:
$$
c(X)= \prod_{i=1}^{\op{dim} X} (1 + \alp_i), \quad
td(X)= \prod_{i=1}^{\op{dim} X}\displaystyle  \frac{\alpha_i}{1-e^{-\alpha_i}}, \quad
L(X)= \prod_{i=1}^{\op{dim} X} \frac{\alpha}{\tanh \alpha},
$$
which are respectively the \emph{Chern class, Todd class} and \emph{$L$-class}.
$T_y(X)$ can be considered as a parameterized Todd class $td(X)$ by $y$. We call this parameterized Todd class $T_{y}(X)$ \emph{the Hirzebruch class of $X$}.

For the distinguished three values $-1, 0, 1$ of $y$, by the definition we have the following:

\begin{itemize}
\item the Euler-Poincar\'e characteristic:

 $\chi(X) = \chi_{-1}(X) = \chi^0(X)- \chi^1(X) + \chi^2(X)- \cdots +(-1)^n\chi^n(X),$

\item the Todd genus:

$\tau (X) = \chi_0(X) = \chi^0(X), \hspace{3cm}$

\item the signature:

$\sigma(X) = \chi_1(X) = \chi^0(X)+ \chi^1(X) + \chi^2(X)+\cdots +\chi^n(X).$
\end{itemize}

\vspace{3pt}
As noted by Hirzebruch in \cite[\S 15.5]{Hir} the following duality formula holds
\begin{equation}\label{duality}
\chi^p(X) = (-1)^n \chi^{n-p}(X).
\end{equation}
This duality formula will play an important role in this paper.

\section{Multiplicativity mod $4$}

If we let
$$\chi^{\op{odd}}(X) = \chi^1(X) + \chi^3(X) + \chi^5(X)  \cdots  \quad \text{the odd part},$$
$$\chi^{\op{even}}(X) = \chi^0(X) + \chi^2(X) + \chi^4(X)  \cdots  \quad \text{the even part}, $$
then we get the following
$$\chi(X) = \chi^{\op{even}}(X)  - \chi^{\op{odd}}(X), \quad \sigma(X) = \chi^{\op{even}}(X)  + \chi^{\op{odd}}(X),$$
from which we have
\begin{equation}\label{eq1}
\sigma(X) + \chi(X) = 2 \chi^{\op{even}}(X), \quad 
\sigma(X) - \chi(X) = 2 \chi^{\op{odd}}(X),
\end{equation}

Thus from either first or second formula of (\ref{eq1})
we get the following: 

\begin{cor}[mod $2$ formula] \label{mod2}
For any fiber bundle $F \hookrightarrow E \to B$ with $F, E, B$ compact complex algebraic manifolds, we have
$$\sigma(E) \equiv \sigma(F)\sigma(B) \, \op{mod} \, 2.$$
\end{cor}
\begin{proof} 
Indeed, using the first formula of (\ref{eq1}), we have
$\sigma(E)+\chi(E) - (\sigma(F\times B) +\chi(F \times B)) \equiv 0 \op{mod} \, 2,$
which becomes
$\sigma(E) - \sigma(F)\sigma(B) \equiv 0 \op{mod} \, 2 $,
since $\sigma(F \times B) = \sigma(F)\sigma(B)$ and $\chi(E) =\chi(F \times B)$. Thus we get the result.

\end{proof}

In \cite{Yo}, using the duality formula (\ref{duality}), we obtained $\sigma(E) \equiv \sigma(F)\sigma(B) \op{mod} 4$ for any fiber bundle $F \hookrightarrow E \to B$ with $F, E, B$ compact complex algebraic manifolds. For the sake of the reader we make a quick review of its proof. (Here we note that if $\op{dim}_{\mathbb C}E$ is odd, then either $\op{dim}_{\mathbb C}F$ or $\op{dim}_{\mathbb C}B$ is odd, thus $\sigma(E) = 0$ and $\sigma(F)\sigma(B) =0$, so $\sigma(E) \equiv \sigma(F)\sigma(B) \op{mod} 4$ holds.) 

Let $\op{dim}_{\mathbb C}X =2n$. Then, using the duality formula (\ref{duality}) we get the following:

$$\chi_y(X) =  \sum_{i=0}^{n-1}\chi^i(X) y^i \Bigl (1+ y^{2n-2i} \Bigr) + \chi^n(X)y^n.$$

Thus we have
$$ \chi(X) = \chi_{-1}(X) = \sum_{i=0}^{n-1}(-1)^i2\chi^i(X)+ (-1)^n \chi^n(X), \quad \sigma(X) =\chi_1(X) = \sum_{i=0}^{n-1} 2\chi^i(X) + \chi^n(X).$$

\begin{equation}\label{eq3}
\sigma(X) + \chi(X) = 2\sum_{i=0}^{n-1} \Bigl(1 + (-1)^i \Bigr) \chi^i(X)+ \Bigl (1 + (-1)^n \Bigr) \chi^n(X),
\end{equation}
\begin{equation}\label{eq4}
\qquad \sigma(X) - \chi(X) = 2\sum_{i=0}^{n-1} \Bigl(1 + (-1)^{i+1} \Bigr) \chi^i(X)+ \Bigl (1 + (-1)^{n+1} \Bigr) \chi^n(X).
\end{equation}

In the case when $n=2k$, (\ref{eq4}) implies that 

\begin{equation}\label{eq3.0}
\sigma(X) - \chi(X) = 4\sum_{j=1}^k \chi^{2j-1}(X),
\end{equation}

which implies the following:

\begin{cor}
For any fiber bundle $F \hookrightarrow E \to B$ such that $\op{dim}_{\mathbb C}E=2n$ with an even integer $n$, we have
$$\sigma(E) \equiv \sigma(F)\sigma(B) \op{mod} 4.$$
\end{cor}
The proof is just like the above case of mod 2 formula.

Next, we consider the case when $n=2k+1$. (\ref{eq3}) implies that
\begin{equation}\label{eq6.-1}
\sigma(X) + \chi(X) = 4\sum_{j=0}^{k} \chi^{2j}(X),
\end{equation}

which implies the following

\begin{cor}
For any fiber bundle $F \hookrightarrow E \to B$ such that $\op{dim}_{\mathbb C}E=2n$ with an odd integer $n$,
$$\sigma(E) \equiv \sigma(F)\sigma(B) \op{mod} 4.$$
\end{cor}

Now we will discuss multiplicativity mod $4$ of $\chi_y$-genera of a fiber bundle $F \hookrightarrow E \to B$ with $F, E, B$ compact complex algebraic manifolds. The following congruence is crucial for both the mod $4$ result and for the mod $8$ result.

\begin{pro}\label{prop}
$$\chi_y(X) \equiv \frac{\sigma(X)}{2}(1+y) + \frac{\chi(X)}{2}(1-y) \,\, \op{mod} \, \,1-y^2,$$
where considering $\op{mod} 1-y^2$ means ``letting $y^2=1$ in the polynomial $\chi_y(X)$".
\end{pro}
\begin{proof}
\begin{align*}
\chi_y(X) & = \sum \chi^i(X)y^i \\
& \equiv \chi^0(X) +\chi^1(X)y + \chi^2(X) +\chi^3(X)y +\chi^4(X) + \chi^5(X)y + \cdots   \,\,\,  \op{mod} \, \,1-y^2\\
& = \chi^{\op{even}}(X) + \chi^{\op{odd}}(X)y  \,\, \, \op{mod} \, \,1-y^2\\
& = \frac{\sigma(X) + \chi(X)}{2} + \frac{\sigma(X) - \chi(X)}{2}y   \,\, \op{mod} \, \,1-y^2  \quad \text{(using (\ref{eq1}))}\\
& =\frac{\sigma(X)}{2}(1+y) + \frac{\chi(X)}{2}(1-y)   \,\, \op{mod} \, \,1-y^2.
\end{align*}
\end{proof}

\begin{rem} The polynomial $1-y^2$ is zero at $y=-1, 1$, for which we have the special values $\chi_{-1}(X) = \chi(X)$ and $\chi_1(X) = \sigma(X)$.
\end{rem}

Now we are ready to prove the following
\begin{thm} \label{mod4-chiy} Let $F \hookrightarrow E \to B$ be a fiber bundle such that $F, E, B$ are compact complex algebraic manifolds. Let $y$ be an odd integer, then
$$\chi_y(E) \equiv \chi_y(F)\chi_y(B) \, \op{mod} \, 4.$$
\end{thm}
\begin{proof}
From Proposition \ref{prop}, we can see that for any fiber bundle $F \hookrightarrow E \to B$ we have
$$\chi_y(E) - \chi_y(F)\chi_y(B) \equiv \frac{\sigma(E) - \sigma(F)\sigma(B)}{2}(1+y) \,\, \op{mod}\,\, 1-y^2.$$

We have shown before that $\sigma(E) \equiv \sigma(F)\sigma(B) \, \, \op{mod} \, \, 4$, i.e. $\sigma(E) -\sigma(F)\sigma(B)$ is divisible by $4$, thus $\displaystyle \frac{\sigma(E) - \sigma(F)\sigma(B)}{2}$ is even. Since $y$ is odd, $1+y$ is even, thus $\displaystyle \frac{\sigma(E) - \sigma(F)\sigma(B)}{2}(1+y) \equiv 0 \,\, \op{mod} \, \, 4$, and $1-y^2 = (1-y)(1+y) \equiv 0 \,\, \op{mod} \, \, 4$. Thus we have
$\chi_y(E) -\chi_y(F)\chi_y(B) \equiv 0 \, \op{mod} \, 4,$
 i.e.  for \emph{any} odd integer $y$ we have
 $$\chi_y(E) \equiv \chi_y(F)\chi_y(B) \, \op{mod} \, 4.$$
\end{proof}

\section{Multiplicativity mod $8$}
We will now investigate multiplicativity modulo $8$ and prove the main results of this note mentioned in the introduction.

\begin{thm}\label{mod8-part1}
 Let $F \hookrightarrow E \to B$ be a fiber bundle such that $F, E, B$ are compact complex algebraic manifolds. Let $y$ be an odd integer, then
\begin{enumerate}
\item If $y \equiv 3 \, \op{mod} \, 4$, then $\chi_y(E) \equiv \chi_y(F)\chi_y(B) \, \op{mod} \, 8.$
\item If $y \equiv 1 \, \op{mod} \, 4$, then $\chi_y(E) \equiv \chi_y(F)\chi_y(B) \, \op{mod} \, 8 \Longleftrightarrow \, \sigma(E) \equiv \sigma(F)\sigma(B) \, \op{mod} \, 8.$
\end{enumerate}
\end{thm}
\begin{proof}Again using Proposition \ref{prop}, we can see that for any fiber bundle $F \hookrightarrow E \to B$ we have
$$\chi_y(E) - \chi_y(F)\chi_y(B) \equiv \frac{\sigma(E) - \sigma(F)\sigma(B)}{2}(1+y) \,\, \op{mod}\,\, 1-y^2.$$
\begin{enumerate}
\item First we observe that if $y$ is odd, then in fact we have $1-y^2 \equiv 0 \,\, \op{mod} \,\, 8.$ Indeed, let $y=2k+1$. Then $1-y^2=(1-y)(1+y) = -2k(2k+2) = -2\cdot 2\cdot k(k+1)$ is divisible by $2 \cdot 2 \cdot 2=8$ because $k(k+1)$ is always even. Let $y \equiv 3 \,\, \op{mod}\, \, 4$, i.e. $1+y$ is divisible by $4$. Then we have $\chi_y(E) - \chi_y(F)\chi_y(B) \equiv \displaystyle  \frac{\sigma(E) - \sigma(F)\sigma(B)}{2}(1+y) \equiv 0 \,\, \op{mod} \, \, 8$. Therefore we have
$$y \equiv 3 \, \op{mod} \, 4 \Longrightarrow \chi_y(E) \equiv \chi_y(F)\chi_y(B) \, \op{mod} \, 8.$$

\item
Let $y \equiv 1 \, \, \op{mod}\, \, 4$, i.e. let $y=4k+1$.
Then $1+y = 4k+2 = 2(2k+1)$. Thus
$$\chi_y(E) -\chi_y(F)\chi_y(B) \equiv \frac{\sigma(E) -\sigma(F)\sigma(B)}{2}(1+y) \equiv 0\,\, \op{mod} \, \,8$$
if and only if
$$\frac{\sigma(E) -\sigma(F)\sigma(B)}{2} \equiv 0 \,\, \op{mod} \, \, 4.$$
I.e.,
$$\chi_y(E) \equiv  \chi_y(F)\chi_y(B)\,\, \op{mod}\,\, 8 \Longleftrightarrow  \sigma(E) \equiv  \sigma(F)\sigma(B)\,\, \op{mod}\,\, 8 $$
\end{enumerate}

\end{proof}

\begin{rem} If we consider $\chi_y(X)$ modulo $y-y^3$, i.e., by letting $y^3 = y$ in $\chi_y(X)$, then we have

$$\chi_y(X) \equiv \tau(X)(1-y^2) + \frac{\chi(X)}{2}(y^2 -y) + \frac{\sigma(X)}{2}(y^2+y) \, \, \op{mod} \,\, y-y^3.$$

In this case we have the following formula:
$$\chi_y(E) -\chi_y(F)\chi_y(B) \equiv (\tau(E) -\tau(F)\tau(B))(1-y^2) + \frac{\sigma(E)-\sigma(F)\sigma(B)}{2}(y^2+y) \,\, \op{mod} \, \, y-y^3.$$

Note that $y-y^3$ has zeros at $y=0, -1,1$ and we have $\chi_0(X) =\tau(X),\chi_{-1}(X) =\chi(X), \chi_1(X) =\sigma(X).$

\end{rem}

In Theorem \ref{mod8-part1} 
(2) we have shown that when $y \equiv 1 \mod{4}$, the $\chi_y$-genera of a fiber bundle have the same multiplicativity behaviour as the signature modulo $8$. This means that we can use the results about the multiplicativity of signature modulo $8$ from \cite{Rov, Rov2} to prove statements for $\chi_y$-genera with $y \equiv 1 \mod{4}$. The first of these statements involves the definition of the Arf invariant of a certain $\zz_2$-quadratic form associated to the fiber bundle. For the convenience of the reader we recall here some relevant definitions and give precise references of where some necessary proofs can be found.

We start by defining a non-singular quadratic form over $\zz_2$ and the Arf invariant.
Let $V$ be a $\bb{Z}_2$-vector space and $\lambda$ a non-singular symmetric bilinear form
$$\lambda: V \otimes V \to \bb{Z}_2,$$
and let $h: V \to \bb{Z}_2$ be a $\bb{Z}_2$-valued quadratic enhancement of this bilinear form which satisfies the following property,
$$h(x+y) = h(x) + h(y) +\lambda (x, y) \in \bb{Z}_2.$$
The Arf invariant was first defined in \cite{Arf} as follows,
\begin{defn} With a symplectic basis $\left\{ e_1, \dots, e_k, \bar{e}_1, \dots, \bar{e}_k   \right\}$ for $V$, the Arf invariant is defined as
$$\textnormal{Arf}(h) = \sum\limits_{j=1}^{k} h(e_j)h(\bar{e}_j) \in \bb{Z}_2.$$
\end{defn}

A \textit{characteristic element} of a symmetric form $(V, \lambda)$ is an element $v \in V$ such that for any $u \in V$ one has:
$$\lambda(u, u) = \lambda(u, v) \in \bb{Z}_2.$$
For example the Wu class $v_{2k}(M) \in H^{2k}(M; \zz_2)$ of a $4k$-dimensional manifold $M$ is a characteristic element of the intersection form.

A \textit{sublagrangian subspace} of a symmetric form $(V, \lambda)$ is a subspace $L$ such that $\lambda(L, L)=0$.

A \textit{Lagrangian subspace} of a symmetric form $(V, \lambda)$ is a subspace $L$ such that $\lambda(L, L)=0$ and $\textnormal{dim} L = \frac{1}{2} \textnormal{dim} V$.

The relation between the signature modulo $8$ of a manifold and the Arf invariant was investigated in \cite[Proposition 2.4.5 and Theorem 4.3.5.]{Rov}. This relation is intricately connected to the $\zz_8$-valued Brown-Kervaire invariant defined in \cite{Bro} and Morita's theorem \cite[Theorem 1.1]{Morita}. Morita's theorem requires the use of the \textit{Pontryagin squares} $\mathcal{P}_2.$ A good reference for this is \cite[Chapter 2]{Mosh-Tang}. The relation states that when the signature of a $4k$-dimensional manifold is divisible by $4$, then modulo $8$ this signature can be expressed as $4$ times the Arf-Kervaire invariant of an associated $\zz_2$-valued quadratic form. The details about how to construct this associated quadratic form are given in \cite[Proposition 2.4.5 and Theorem 4.3.5.]{Rov} and the construction for the case of a fiber bundle is given in Theorem \ref{Arf-chiy}.

\begin{thm}\label{Arf-chiy} Let $F \hookrightarrow E \to B$ be a fiber bundle such that $F, E, B$ are smooth compact complex algebraic varieties. 
If $y \equiv 1 \, \op{mod} \, 4$, then
$$\chi_y(E) - \chi_y(F) \chi_y(B) \equiv 4 \textnormal{Arf}(W, \mu, h)  \pmod{8}.$$
where $(W, \mu, h)= \left( L^{\perp}/L , [\lambda \oplus - \lambda'], \left[ \mathcal{P}_2 \oplus-\mathcal{P}'_2 \right]/2 \right)$ with
\begin{itemize}
\item $L = \langle v_{2k} \rangle \subset L^{\perp}$, with $v_{2k}=v_{2k}(E) \oplus v_{2k}(F \times B) \in H^{2k}(E;\bZ_2) \oplus H^{2k}(F \times B;\bZ_2)$ the Wu class of $E \sqcup F \times B$,
\item $L^{\perp} =\left\{(x, x') \in H^{2k}(E;\bZ_2) \oplus H^{2k}(F \times B;\bZ_2) \,\vert\, \lambda(x,x) =  \lambda'(x', x') \in \bZ_2 \right\}$,
\item $\mathcal{P}_2$ and $\mathcal{P'}_2$ are the Pontryagin squares of $E$ and $F \times B$ respectively.
\end{itemize}
\end{thm}
\begin{proof}This theorem is a direct consequence of \cite[Theorem 6.2.1]{Rov} and Theorem \ref{mod8-part1}.
\end{proof}

The other statement which follows from the work on the signature modulo $8$ in \cite{Rov} and \cite{Rov2} is as follows:
\begin{thm} Let $F \hookrightarrow E \to B$ be a fiber bundle such that $F, E, B$ are smooth compact complex algebraic varieties with $\textnormal{dim}_{\bR} F =2m$. If the action of $\pi_1(B)$ on $H^m(F, \bZ)/torsion \otimes \bZ_4$ is trivial, then for \emph{any} odd integer $y$
$$\chi_y(E) \equiv \chi_y(F)\chi_y(B) \, \op{mod} \, 8.$$
\end{thm}
\begin{proof}
In \cite[Theorem 6.3.1]{Rov} it is shown that if the action of $\pi_1(B)$ on $H^m(F, \bZ)/torsion  \otimes \bZ_4$ is trivial, then
$$\sigma(E) \equiv \sigma(F) \sigma(B)  \, \op{mod} \, 8.$$
For a more succinct proof of this result see \cite[Theorem 3.5]{Rov2}.
Combining this  with Theorem \ref{mod8-part1} 
(2) we obtain
$$\chi_y(E) \equiv \chi_y(F)\chi_y(B) \, \op{mod} \, 8.$$
\end{proof}

\section {Concluding remarks}

\begin{rem}
As to the congruence formulae for $\chi_y$, above we consider the case when $y$ is an odd integer. When it comes to the case when $y$ is an even integer, we have not found any interesting congruence formula.
\end{rem}

\begin{rem}
 In \cite{CLMS2} Cappell, Libgober, Maxim and Shaneson obtain the following Atiyah-Meyer type formula:
$$\chi_y(E) = \int_B ch^*(\chi_y(\pi))\cup \widetilde T_y^*(TB),$$
where $ch^*$ is the Chern character, $\chi_y(\pi)$ is the $K$-theory $\chi_y$-characteristic of the bundle projection map $\pi:E \to B$ and $\widetilde T_y^* (TB)$ is the unnormalized Hirzebruch class. For more detailed explanation of these see \cite{CLMS2}. Since $ch^0(\chi_y(\pi)) = \chi_y(F)$, as explained in \cite{CLMS2}, the right-hand-side of the above Atiyah-Meyer type formula is
$$\int_B ch^*(\chi_y(\pi))\cup \widetilde T_y^*(TB) = \chi_y(F)\chi_y(B) + \text{correction terms},$$
where the correction terms measure the deviation from multiplicativity.
Hence it follows from our results above that

\begin{enumerate}
\item for any odd integer $y$ the expression given by the correction
terms is divisible by $4$,
\item if $y \equiv 3 \, \op{mod} \, 4$, the expression given by the correction
terms is divisible by $8$,
\item if $y \equiv 1 \, \op{mod} \, 4$, the expression given by the correction
terms is divisible by $8$ if and only if
$\sigma(E) \equiv \sigma(F)\sigma(B) \op{mod} \, 8$.
\end{enumerate}
It remains to be seen if one could get the above results directly from the above Atiyah-Meyer type formula.
\end{rem}

\begin{rem} Even if $X$ is singular, we can define $\chi_y(X)$ (using the same symbol) using the mixed Hodge structure (e.g. see \cite{BSY1, CMS1, CLMS2, MS}).
We define 
$\chi_y(X):= \sum_{i,  p\geq 0} (-1)^i \op{dim}_{\bC} Gr^p_{\mathcal F}(H^i_c(X,  \bC)) (-y)^p,$ where $\mathcal{F}$ is the Hodge filtration of the mixed Hodge structure of $X$.
In this case we can consider the above congruences even for fiber bundles $F \hookrightarrow E \to B$ with $F, E, B$ being possibly singular. In this case we have the following results, $\op{mod} \, 4$ and $\op{mod} \, 8$ being replaced by respectively $\op{mod} \, 2$ and $\op{mod} \, 4$:
\begin{enumerate}
\item For any odd integer $y$, $\chi_y(E) \equiv \chi_y(F)\chi_y(B) \, \op{mod} \,2.$
\item If $y \equiv 3 \, \op{mod} \, 4$, then $\chi_y(E) \equiv \chi_y(F)\chi_y(B) \, \op{mod} \, 4.$
\item If $y \equiv 1 \, \op{mod} \, 4$, then $\chi_y(E) \equiv \chi_y(F)\chi_y(B) \, \op{mod} \, 4 \Longleftrightarrow \, \sigma^H(E) \equiv \sigma^H(F)\sigma^H(B) \, \op{mod} \, 4.$
Here $\sigma^H(X) := \chi_1(X)$ is called the Hodge-signature of $X$, which is equal to the usual signature when $X$ is nonsingular and compact.
\end{enumerate}
However, if the duality formula (\ref{duality}) in \S 2 still holds even in the singular case\footnote{A projective simplicial toric variety as discussed in \cite{MS2} (cf. \cite{MSS}, \cite{Sch}) is such a variety (pointed out by L. Maxim and J. Sch\"urmann).}
, then the same results as in the smooth case hold, i.e., in the above formulas $\op{mod} \, 2$ and $\op{mod} \, 4$ are changed back to $\op{mod} \, 4$ and $\op{mod} \, 8$ respectively. We will deal with the singular case and the intersection homology $\chi_y$-genus $I\chi_y(X)$ in a different paper.

\end{rem}

\begin{rem} It seems that all these multiplicativity mod $4$ and mod $8$ properties for genera
of compact complex algebraic manifolds are obtained by taking degrees of appropriate
characteristic class formulae for the pushforward under the bundle projection
map of (homology) characteristic classes of the total space. This is indeed the case
for the usual multuplicativity and for the correction terms (e.g., see works of Banagl, Cappell, Libgober, Maxim, Sch\"urmann, Shaneson \cite{BCS, CS, CMS1, CLMS2, MS} etc.). Thus it is also reasonable to generalize 
such multiplicativity formulae of characteristic homology classes to the singular setting (e.g., by making use of intersection homology). 
We would like to deal with this problem as well in a different paper.
\end{rem}

\begin{rem} In this paper the congruence formula (\ref{prop}) of two integral polynomials is a key. For $a(y), b(y) \in \mathbb Z[y]$, the congruence $a(y) \equiv 0 \, \op{mod} \, b(y)$ of course means that $\exists c(y) \in \mathbb Z[y]$ such that $a(y) = b(y)c(y)$. Then for any integer $n \in \mathbb Z$ we have
$a(n) =b(n)c(n)$, i.e. $a(n) \equiv 0 \, \op{mod} \, c(n)$. (In our case we consider only odd integers $n$, though.) Namely, we have
$$a(y) \equiv 0 \, \op{mod} \, b(y) \, \, (\text{in} \, \, \mathbb Z[y]) \, \Longrightarrow \forall n \in \mathbb Z, \,\, a(n) \equiv 0 \, \op{mod} \, b(n) \, \, (\text{in} \, \, \mathbb Z).$$
It should be noted that the converse of this implication does not necessarily hold. A counterexample is given by Fermat's little theorem. Indeed, let $p$ be a prime number and $a(y) = y^p-y, b(y)=p$. Fermat's Little Theorem says that for any integer $n \in \mathbb Z$ $n^p \equiv n \, \op{mod} \, p$, i.e. for $\forall n \in \mathbb Z$ $a(n)=n^p-n \equiv 0 \, \op{mod}\, p$. However clearly $a(y)=y^p-y \not \equiv 0 \, \op{mod} \, p  \, \, (\text{in} \, \, \mathbb Z[y])$. The case when $p=5$ is pointed out in \cite{CV}, where L.F. C\'aceres and J. A. V\'elez-Marulanda consider some special cases when this converse does hold.

\end{rem}
\bigskip

\noindent
{\bf Acknowledgements:} \,
We would like to thank Sylvain Cappell, Laurentiu Maxim, Andrew Ranicki and J\"org Sch\"urmann for useful comments. We also would like to thank the referee for his/her thorough reading of the paper and useful comments and suggestions. 

We both had opportunities to give talks at ``Workshop on Stratified Spaces:Perspectives from Analysis, Geometry and Topology, August 22 - 26", in ``Focus Program on Topology, Stratified Spaces and Particle Physics, August 8 - 26, 2016, The Fields Institute for Research in Mathematical Sciences". We would like to thank the organizers (M. Banagl, E. Bierstone, S. Cappell, L. Maxim and T. Weigand) of the workshop and the staff of The Fields Institute for the wonderful organization and for a wonderful atmosphere to work.



\begin{thebibliography}{99}

\bibitem{Arf}
C. Arf, {\it Untersuchungen \"uber quadratische Formen in K\"orpern der
 Charakteristik 2. {I}}, in Journal f\"ur die Reine und Angewandte Mathematik, {\bf 183} (1941) 148-167

\bibitem{At}
M. F. Atiyah, {\it The signature of fiber-bundles},
in ``Global Analysis (Papers in Honor of K. Kodaira)”, Univ. Tokyo Press, Tokyo (1969) 73-84

\bibitem{BCS} M. Banagl, S. Cappell and J. Shaneson, {\it Computing twisted signatures and $L$-classes of stratified sapces}, Math. Ann. {\bf 326} (2003), 589--623.

\bibitem{BSY1}
J.-P. Brasselet, J. Sch\"urmann and S. Yokura, {\it Hirzebruch classes and motivic Chern classes for singular spaces},  Journal of Topology and Analysis, {\bf 2}, No.1 (2010), 1--55 .

\bibitem{Bro}
E. H. Brown, Jr., {\it Generalizations of the Kervaire invariant},
Ann. of Math. {\bf 95} (2) (1972), :68-383,

\bibitem{CV}
L. F. C\'aceres and J. A. V\'elez-Marulanda, {\it
Some Divisibility Properties in Ring of Polynomials over a UFD}, arXiv:0806.1398v1 [math.AG],

\bibitem {CMS1}
S. E. Cappell, L. Maxim and J. Shaneson, {\it Hodge genera of algebraic varieties, I.},
Comm. Pure Appl. Math. {\bf 61} (2008), no. 3, 422-449.

\bibitem {CLMS1}
S. E. Cappell, A. Libgober, L. Maxim.  and J. Shaneson,
{\it Hodge genera and characteristic classes of complex algebraic varieties},
Electron. Res. Announc. Math. Sci. {\bf 15} (2008), 1-7.

\bibitem {CLMS2}
S. E. Cappell, A. Libgober, L. Maxim.  and J. Shaneson,
{\it Hodge genera of algebraic varieties, II.},
Math. Ann. {\bf 345} (2009), 925-97.

\bibitem{CS} S. Cappell and J. Shaneson, {\it Stratifiable maps and topological invariants}, J. Amer. Math. Soc., {\bf 4} (1991), 521--551.

\bibitem{CHS}
S. S. Chern, F. Hirzebruch, J.-P. Serre, {\it On the index of a fibered manifold},
Proc. Amer. Math. Soc. {\bf 8} (1957) 587-596.


\bibitem{HKR}
I. Hambleton, A. Korezeniewski and A. Ranicki,
{\it The signature of a fiber bundle is multiplicative mod 4},
Geometry \& Topology, {\bf 11} (2007), 251-314.

\bibitem{Hir0}
F. Hirzebruch, {\it The signature of ramified coverings}, in ``Global Analysis (Papers in Honor of K. Kodaira)”, Univ. Tokyo Press, Tokyo (1969) 253-265.

\bibitem{Hir}
F. Hirzebruch, {\it Topological Methods in Algebraic Geometry, 3rd ed. (1st German ed. 1956)},
Springer-Verlag, 1966.

\bibitem{HBJ}
F. Hirzebruch, T. Berger and R. Jung, {\it Manifolds and Modular Forms},
Vieweg, 1992.

\bibitem{Ko}
K. Kodaira, {\it A certain type of irregular algebraic surfaces}, J. Analyse Math. {\bf 19} (1967)
207-215.

\bibitem{MSS}
L. Maxim, M. Saito and J. Sch\"urmann, {\it 
Symmetric products of mixed Hodge modules}, J. Math. Pures Appl. {\bf 96} (2011), 462-483.

\bibitem{MS}
L. Maxim and J. Sch\"urmann, {\it Hodge-theoretic Atiyah-Meyer formulae and the stratified multiplicative property}, in ``Singularities I: Algebraic and Analytic Aspects", Contemp. Math. {\bf 474} (2008),145-167.


\bibitem{MS2}
L. Maxim and J. Sch\"urmann, {\it Characteristic classes of singular toric varieties},
Comm. Pure Appl. Math. {\bf 68} (2015), 2177-2236.

\bibitem{Morita}
 S. Morita, {\it On the Pontrjagin square and the signature}, in ``J. Fac. Sci. Univ. Tokyo Sect. IA Math." {\bf 18} (1971), 405--414.


\bibitem{Mosh-Tang}
R.E. Mosher and M.C. Tangora, {\it Cohomology operations and applications in homotopy theory}, in ``Harper \& Row Publishers". (1968)


\bibitem{Rov}
C. Rovi, {\it The Signature Modulo $8$ of fiber Bundles}, available at http://arxiv.org/pdf/1507.08328v1.pdf, Ph.D Thesis (Univ. of Edinburgh, UK), 2015.

\bibitem{Rov2}
C. Rovi, {\it The non-multplicativity of the signature modulo 8 of a fiber bundle is the Arf-Kervaire invariant}, arXiv:1609.01365v2 [math.AT] , to appear in Algebraic \& Geometric Topology.

\bibitem{Sch}
J. Sch\"urmann, {\it Characteristic classes of mixed Hodge modules}, in ``Topology of Stratified Spaces",  MSRI Publications, Vol. {\bf 8} (2011), 419-470.


\bibitem{Yo}
S. Yokura, {\it Hirzebruch $\chi_y$-genera of complex algebraic fiber bundles
--- the multiplicativity of the signature modulo $4$ ---}, arXiv:1601.04629v4 [math.AG], to appear in IMPANGA Lecture Notes ``Vector bundles, Schubert varieties, and equivariant cohomology", Birkh\"auser (2017).

\end{thebibliography}
\end{document}